\documentclass[11pt]{amsart}
\usepackage{geometry}                
\usepackage{graphicx}
\usepackage{amssymb}
\usepackage{epstopdf}

\numberwithin{equation}{section}
\usepackage[colorlinks=true]{hyperref}
\usepackage{enumerate}
\hypersetup{urlcolor=blue, citecolor=blue}

\newtheorem{theorem}{Theorem}[section]
\newtheorem{corollary}{Corollary}[section]

\newtheorem{lemma}[theorem]{Lemma}
\newtheorem{proposition}[theorem]{Proposition}

\theoremstyle{definition}
\newtheorem{definition}[theorem]{Definition}

\newcommand{\rn}{\mathbb R^N}
\newcommand{\R}{\mathbb R}

\newcommand{\Hm}{{\mathcal H}^{N-1}}

\newcommand{\h}{\mathcal{H}^{n-1}}

\newcommand{\Ao}{A_{opt}}

\title[Quantitative Faber-Krahn inequality for the Robin Laplacian]{The quantitative Faber-Krahn inequality for the  Robin Laplacian}

\author[D. Bucur, V. Ferone, C. Nitsch, C. Trombetti]{Dorin Bucur, Vincenzo Ferone, Carlo Nitsch, Cristina Trombetti}
\date{}                                           
\address{\vskip1cm\noindent Dorin Bucur \hfill\break\vskip-.2cm
\noindent Institut Universitaire de France and Laboratoire de Math\'ematiques, CNRS UMR 5127,
Universit\' e Savoie Mont Blanc, Campus Scientifique, 73376 Le-Bourget-Du-Lac, France. \hfill\break\vskip-.2cm
\noindent e-mail: {\tt dorin.bucur@univ-savoie.fr}
\hfill\break\vskip-.2cm
\noindent Vincenzo Ferone, Carlo Nitsch, Cristina Trombetti\hfill\break\vskip-.2cm
\noindent Dipartimento di Matematica e Applicazioni ``R. Caccioppoli'', Universit\`{a}
degli Studi di Napoli ``Federico II'', Complesso Universitario Monte S. Angelo, via Cintia
- 80126 Napoli, Italy. \hfill\break\vskip-.2cm
\noindent e-mail: \tt ferone@unina.it; c.nitsch@unina.it;
cristina@unina.it}


\begin{document}

\maketitle

\begin{abstract}
We prove a quantitative form of the Faber-Krahn inequality for the first eigenvalue of the Laplace operator with Robin boundary conditions. The asymmetry term involves the square power of the Fraenkel asymmetry, multiplied by a constant depending on the Robin parameter, the dimension of the space and the measure of the set. 
\end{abstract}

\section{Introduction and Main result}

Let  $\beta >0$ and $\Omega \subseteq \R^N$ be an open, bounded, smooth, connected  set.     The first eigenvalue of the Laplace operator on $\Omega$ with Robin boundary conditions conditions is denoted by $\lambda_{1,\beta}(\Omega)$ and is the unique positive number for which the following equation has a  solution $u$ of constant sign
\begin{equation}\label{eigen_problem}
\left\{
\begin{array}{ll}
-\Delta u= \lambda_{1,\beta}(\Omega)  u& \mbox{in $\Omega$}\\\\
\dfrac{\partial u}{\partial \nu} +\beta \,u =0 & \mbox{on $\partial\Omega$.}
\end{array}
\right.
\end{equation}
Above $\nu$  denotes the outer unit normal to $\partial \Omega$. If $\Omega$ is not smooth, an appropriate definition of the eigenvalue is given in terms of minimal values of the associated  Rayleigh quotient, see Section 2.

The Faber-Krahn inequality, proved by Bossel  \cite{Bo2} in two dimensions and by Daners \cite{D1} in any dimension of the space, states among all Lipschitz domains with given volume, balls minimize the first Robin eigenvalue. This inequality has been extended to arbitrary open sets in \cite{BG10}. 
To be more precise, let $\Omega\subset \R^N$ be an open set of finite measure and $B$  a ball of the same volume as $\Omega$. Then
\begin{equation}
\label{FK}
\lambda_{1,\beta}(\Omega) \ge \lambda_{1,\beta}(B).
\end{equation}

The purpose of the present paper is to refine this result with a quantitative estimate, proving that the difference between the two terms in \eqref{FK}  controls some distance
of $\Omega$ to a ball with the same volume.

 The {\sl Fraenkel asymmetry}  of $\Omega$ is defined by (see for instance \cite{FMP08})
\begin{equation}
\label{fraenkel}
\mathcal A (\Omega)= \min \Big \{  \frac{|\Omega \Delta (x_0+B)|}{|B|}: x_0\in \R^N\Big \},
\end{equation}
where the symbol $\Delta$ stands for the symmetric difference between sets.

\medskip

The main result of this paper is the following quantitative inequality. 
\begin{theorem}\label{main}
Let $\Omega$ be an open set with finite measure. 
There exists a positive constant $c$ depending only on $N, \beta, |\Omega|$, such that 
\begin{equation}
\label{bfnt01}
\lambda_{1,\beta}(\Omega) - \lambda_{1,\beta}(B) \ge c  \mathcal A^2 (\Omega).
\end{equation}
\end{theorem}
The power $2$ on the Fraenkel asymmetry is sharp. The constant $c$ is explicit but, as we shall explain in the last section, is not optimal.

The resolution by  Fusco, Maggi and Pratelli \cite{FMP08}   of the conjecture of Hall was an important breakthrough on the quantitative isoperimetric inequality and opened the way for a class of quantitative inequalities of spectral type. Two new proofs of the quantitative isoperimetric inequality   followed in 2010 and 2012 (see \cite{FMP10} and \cite{CL12}). 

Quantitative estimates involving the lower eigenvalues of the Laplace operator were naturally investigated. For the first non zero eigenvalue of the Neumann Laplacian (see \cite{BP12}), or the Steklov problem (see \cite{BDPR12}), these estimates can be obtained quite directly, since the balls are maximizers and the quantitative estimate has to be carried on fixed test functions. 

The same question for the Dirichlet Laplacian turned out to be much more difficult. Contrary to the maximization problems, the nature of this one is similar to the classical isoperimetric inequality. We refer to \cite{FMP09} (and the references therein)  for a first general quantitative estimate. The problem was completely solved by Brasco, De Philippis and Velichkov in \cite{BDe}, by techniques involving the regularity of free boundary problems of Alt-Caffarelli type. Indeed, the inequality proved in \cite{BDe} for the first eigenvalue of the Dirichlet Laplacian (formally $\beta =+ \infty$) reads
$$|\Omega|^\frac 2N \lambda_{1,\infty}(\Omega)-|B|^\frac 2N \lambda_{1,\infty}(B)\ge c A^2 (\Omega),$$
where the constant $c$ depends only on the dimension of the space, and the inequality itself is scale invariant. 

The recent survey by Brasco and De Philippis \cite{BDP16} gives a complete and up to date overview of the topic. The quantitative estimate for the Robin eigenvalue is left as an open problem, and the purpose of our paper is to solve it.

 Contrary to the Dirichlet case,  scale invariance is not preserved in the Robin problem because of the presence of the parameter $\beta$. For this reason, the constant appearing on the right hand side in \eqref{bfnt01} depends on  the measure of the domain and the parameter $\beta$. In order to emphasize the scale dependence problem, we refer to the paper of \cite{Ke10}, where it is proved that the optimal set minimizing the third eigenvalue of the Robin Laplacian among sets of volume equal to 1 can not have the same shape (up to rigid transformations) for all values of $\beta$. We also point out that, contrary to the result for the Dirichlet problem in \cite{BDe}, our constant is  explicit. Meanwhile, our constants (depending on $\beta$, at fixed measure) are not optimal. Nevertheless, finding better constants in our approach for $\beta \rightarrow +\infty$ could lead to a new proof of the quantitative estimate in  \cite{BDe}.

Our proof is based on the use of the selection principle (see references \cite{CL12}, \cite{BDe}) which consists in dealing only with sets having a small Fraenkel asymmetry and to modify them into new sets which have a comparable Fraenkel asymmetry, lower eigenvalues, and enjoy some other {\it regularity} properties. Contrary to the results of \cite{CL12} and \cite{BDe} where by {\it regularity} property one understands {\it small distance in $C^1$-norm to the ball}, in our case the property we search involves the behavior of an eigenfunction near the boundary. 

In order to make the selection principle work, one has to solve a free discontinuity problem. The prototype of a free discontinuity problem is the Mumford-Shah functional. Nevertheless, in the recent years much progress has been done in the study of free discontinuity problems with Robin boundary conditions. We refer the reader to \cite{BG10,BG15,BG15AIHP,BL14,CK}. Solving such problems requires a quite involved background in the theory of special functions of bounded variation. Our intention is to keep technicalities at the lowest possible level and to make reference to the original papers for all technical results. 

\medskip

\noindent {\bf Strategy of the proof:}

\begin{itemize}
\medskip
\item {\bf Step 1.} Let  $\Omega\subset \R^N$ be  an open set of finite volume. For a suitable value $k >0$ (which depends only on $|\Omega|$ and $N$), we solve the free discontinuity problem
$$\min \{\lambda_{1, \beta} (A)+k|A| : A \subset \Omega, A \mbox{ open}\}.$$
We prove that a solution exists and that for a nonnegative, $L^2$-normalized eigenfunction $u$ on $A$, we have
$$ess \inf_{x\in A} u(x) \ge \alpha,$$
where the value $\alpha$ depends only on $N, \beta, k, \lambda_{1,\beta} (\Omega), |\Omega|$ (on the nonsmooth set $A$, the fundamental Robin eigenvalue $\lambda_{1, \beta} (A)$  is defined by \eqref{weak}, next section).

\medskip

\item {\bf Step 2.} Let $A$ be a connected,  open set with finite measure, and $B$ a ball of the same measure. We prove that for a nonnegative, $L^2$-normalized eigenfunction $u$ we have
$$\lambda_{1, \beta} (A)-\lambda_{1, \beta} (B)\ge   \frac \beta2 (ess \inf_{x\in A} u(x))^2\big  ({\mathcal H}^{N-1}(\partial^* A)-{\mathcal H}^{N-1}(\partial B)\big ).$$
Above, $\partial ^* A$ stands for the reduced boundary of $A$ and the convention $0\cdot \infty =0$ is used. In particular, this implies that $ess \inf_{x\in A} u(x)=0$ as soon as $A$ does not have finite perimeter. \medskip

\item {\bf Step 3.} We apply Step 2 to the set $A$ issued from Step 1. We use the quantitative isoperimetric inequality \cite{FMP08} for the set  $A$ and we prove that the Fraenkel asymmetry of $A$ and the loss of measure $|\Omega\setminus A|$ are comparable with the Fraenkel asymmetry of $\Omega$. Consequently we get inequality \eqref{bfnt01}.
\end{itemize}
\medskip

 Step 1 is the most technical part and is related to regularity results for free discontinuity problems of Robin type. We refer the reader to the recent paper \cite{BG15AIHP} which fixes the framework of such problems and to \cite{BG10,BL14, BG15} for several results that we use in this paper. In particular, the sharp estimate of the constant $\alpha$ is obtained using ideas from \cite{CK}. The proof of Step 2 is based on a refinement of the proof of Bossel and Daners. Step 3 is a direct consequence of the quantitative estimate for the perimeters, proved in \cite{FMP08}. The inclusion $A \subset \Omega$ and of the uniform convexity of the mapping 
$$(0,+\infty)\ni R\mapsto \lambda_{1, \beta} (B_R)+ k|B_R|,$$
 where $B_R$ stands for the ball of radius $R$, insure that the Fraenkel asymmetries of $A$ and $\Omega$ are comparable.

It is important to point out that the Faber-Krahn inequality for the Robin-Laplacian was extended in \cite{BG10} to arbitrary open sets and, even more, transformed in a Poincar\'e like inequality with trace term for special functions of bounded variation. The natural question 
\medskip

\centerline {\it What is the precise  definition of the Robin eigenvalue in an arbitrary open set ?}  
\medskip

\noindent will be answered in Section 2.
Defining the Robin eigenvalue problem in an arbitrary open set is not straightforward and requires some technicalities developed in the theory of special functions of bounded variation, as done in \cite{BG10,BG15}. One would be happy to have a quantitative estimate in the (smaller) class of Lipschitz sets and to avoid these technicalities. Nevertheless, our proof of the quantitative estimate is fundamentally requiring several regularity results for free discontinuity problems of Robin type, so that those technicalities are hardly avoidable. Readers who are not familiar with this theory, should imagine (from an intuitive point of view) that all sets are Lipschitz.

\section{The Robin eigenvalue problem}
Throughout the paper, $B$ denotes a ball with specified size. If ambiguity occurs, we denote by $B^{|A|}$  a ball with the same measure as the set $A$ and by $B_r$ a ball of radius $r$. 

\noindent {\bf The fundamental Robin eigenvalue in a Lipschitz set.} Let $\beta >0$ be fixed. Let $\Omega\subset \R^N$ be a bounded, open, Lipschitz, connected set. The fundamental eigenvalue of the Robin-Laplacian on $\Omega$ is defined by 
\begin{equation}
\label{classical}
\lambda_{1,\beta} (\Omega) = \mathop{\min_{v\in H^1(\Omega)}}_{ v\ne 0} \dfrac {\displaystyle \int_\Omega |\nabla v|^2 \, dx + \beta \displaystyle\int_{\partial \Omega} v^2 \, d \Hm }{\displaystyle\int_\Omega  v^2 \, dx}.
\end{equation}
The minimizer $u$ can be chosen to be nonnegative (see for instance \cite{D1}) and satisfies the equation \eqref{eigen_problem}  in the weak form
$$\forall v \in H^1(\Omega)\;\; \int_\Omega \nabla u \nabla v dx +\beta \int_{\partial \Omega} uv d \Hm = \lambda_{1,\beta} (\Omega) \int_\Omega uv dx. $$

\medskip

\noindent {\bf The Robin eigenvalue on balls.} We denote below by $B_r$ the ball of radius $r$.  We recall that on a ball, the eigenfunction associated to the first eigenvalue is radially symmetric. 

 The following result plays a crucial role in the estimate of the quantitative term. It was observed in  \cite[Section 5]{BG15} (see in particular  Remark 5.5), that the mapping 
$r \mapsto \lambda_{1,\beta} (B_r)$ is  convex. 
We give below a short proof of a stronger assertion. 

\begin{lemma}
\label{convexity}
The function $g:  r \in (0,+\infty) \rightarrow r\,\lambda_{1,\beta} (B_r)$ is strictly convex.
\end{lemma}
\begin{proof}
The eigenvalue $\lambda_1(B_r)$ is implicitely defined by the equation $G(r, \lambda)=0$ where
\[
G(r, \lambda)= \sqrt{\lambda} J_{N/2}(\sqrt{\lambda} r) - \beta J_{N/2-1}(\sqrt{\lambda} r)
\]
Properties of Bessel functions see \cite{AS} give
\[
\frac{d \lambda}{dr}= - \frac{G_r}{G_{\lambda}} = -\frac{2 \lambda}{r} \left( \frac{\lambda r -(N-1)\beta +\beta^2 r}{\lambda r -(N-2)\beta +\beta^2 r}\right).
\]

and then
\begin{equation}
\label{deriv}
\left \{
\begin{array}{ll}
g'(r) =- \displaystyle\frac{g(r)}{r} \left(1 - \displaystyle\frac {2 \beta}{g(r)-(N-2)\beta +\beta^2 r}\right)\\
g(0)= N \beta.
\end{array}
\right.
\end{equation}
Moreover, by scaling,  we have that $g'(r)<0$ and then $g(r) <N \beta$. 
Denoted by $h(r) = g(r)-(N-2)\beta +\beta^2 r$ and differentiating \eqref{deriv}  
\[
g''(r)= \frac{2g(r)}{r^2}\left(1- \frac{2\beta}{h(r)}\right) + \frac{2\beta g(r)}{r^2 h^2(r)}(-h(r)+2\beta -\beta^2r)- \frac{2\beta g(r)g'(r)}{h^2(r)r} >0
\]
since :
\begin{itemize}
\item $\frac{2g(r)}{r^2}\left(1- \frac{2\beta}{h(r)}\right)= -\frac{2g'(r)}{r} > 0$;
\smallskip
\item $(-h(r)+2\beta -\beta^2x)= N\beta - g(r) >0$;
\smallskip
\item $- \frac{2\beta g(r)g'(r)}{h^2(r)r}>0$
\end{itemize}
the strict convexity of $g(r)$ follows at once.
\end{proof}
An important consequence is the following.
\begin{corollary}
\label{uniformconvexity}
For every $k>0$, the function $  r \in (0,+\infty) \rightarrow \lambda_{1,\beta}(B_r) + k|B_r|$ is uniformly convex.
\end{corollary}
\begin{proof}
The proof is an immediate consequence of Lemma \ref{convexity}.
\end{proof}
A second consequence concerns the penalized version of the Faber-Krahn inequality (see   \cite[Section 5]{BG15} for a discussion of this topic).
\begin{corollary}
\label{bfnt08}
Let $m>0$ be given and $B$ a ball of measure $m$. There exists a unique $k >0$ such that for every open set with finite measure $A\subset \R^N$ we have
\[
\lambda_{1,\beta}(A) + k |A| \ge \lambda_{1,\beta}(B) + k |B|.
\] 
\end{corollary}

\noindent{\bf The fundamental Robin eigenvalue in nonsmooth open sets.}
Let $A \subset \R^N$ be an open set of finite measure. Since the boundary of $A$ is not assumed to be smooth,  the trace of a Sobolev function from $H^1(A)$ is not well defined on $\partial A$.  In view of the Rayleigh quotient formulation \eqref{classical}, the Robin eigenvalue problem on $A$ can not be defined.

We follow the framework introduced in \cite{BG10}, which gives a natural definition for the fundamental eigenvalue in every open set. This definition  extends the classical one on Lipschitz sets by replacing Sobolev functions in the Rayleigh quotient, by functions of bounded variation.

\medskip

Below, $SBV(\R^N)$ stands for the space of special functions of bounded variation on $\R^N$. For a function $u\in SBV(\R^N)$,  $\nabla u$ denotes  the approximate gradient, $u^+, u^-$ denote the approximate upper and lower limits, and $J_u=\{x \in \R^N: u^+(x)>u^-(x)\}$ is the jump set. We refer the reader to \cite{AFP00} for a detailed analysis of the  $SBV$ spaces. 

In \cite{BG10} it is introduced the space $SBV^\frac 12(\R^N)$ of nonnegative measurable functions $u$ such that $u^2\in SBV(\R^N)$. As $u^2 \in SBV(\R^N)$, one defines 
\begin{itemize}
\item the set $J_u$ is defined to be the set $J_{u^2}$
\item the jump terms $u^\pm$ are defined as $\sqrt {(u^2)^\pm}$ 
\item the gradient of $u$ is defined by $\nabla u= \frac{\nabla u^2}{2 u} 1_{\{u>0\}}$. 
\end{itemize}

\begin{definition}\label{bfnt07}
Let $A \subset \R^N$ be an open set of finite measure. 
The fundamental Robin eigenvalue of the set $A$ is by definition
\begin{equation}
\label{weak}
\lambda_{1,\beta} (A) =    \min \left \{ \begin{array}{ll}\displaystyle& \frac{  \displaystyle\int_{\rn} |\nabla u|^2 \, dx + \beta  \displaystyle\int_{J_u}( |u^+|^2+|u^-|^2 )\, d\h}{ \displaystyle\int_{\rn} u^2 \, dx}: u \in SBV^{  \frac 12} (\rn);  \\
\cr
& u=0 \> \hbox{a.e. in} \> \rn \setminus A; J_u \subset \partial A  \>\hbox{a.e.}  \mathcal{H}^{N-1}. \end{array} \right \}
\end{equation}
\end{definition}
\medskip
Here are some important facts about this definition. 
\begin{itemize}
\item The minimum in \eqref{weak} is attained: see \cite[Proposition 1, Section 3]{BG10}.
\item All minimizers are called eigenfunctions and they satisfy
$$\forall a.e. \; t>0, \; \forall v\in SBV^\frac 12(\R^N)\cap L^\infty(\R^N), J_v \subset J_u, v=0\; a.e. \;on\; \{u \le t\}$$

$$\mbox{such that } \int_{J_u \setminus J_v} |v|^2 d \Hm <+\infty, $$
$$ \int_{\R^N} \nabla u \nabla v  \, dx+ \beta \int_{J_u} u^+\gamma^+(v)+u^-\gamma^-(v)  \, d\Hm = \lambda_{1,\beta} (A) \int_{\R^N} uv \, dx,$$
where the approximate limits $v^\pm(x)$ are relabeled by  $\gamma^\pm(v)$ in such a way to have the same orientation as the approximate limits of $u$ at jump points. We refer to \cite[Proposition 2, Section 6]{BG10} and \cite[Theorem 6.11]{BG15} for this assertion.
\item If $A$ is disconnected, there exist eigenfunctions vanishing on all, but one, connected components.
\end{itemize}

If $A$ is bounded and Lipschitz, the Rayleigh quotient introduced in Definition \ref{bfnt07} and the classical one \eqref{classical} produce the same eigenvalues  and eigenfunctions.

We give the following result.
\begin{lemma}\label{bfnt05}
Let $A \subset \R^N$ be an open set of finite measure. There exists a constant $C$ depending only on $N, \beta, \lambda_{1,\beta}(A)$ and $|A|$  such that for every   minimizer $u$ in \eqref{weak} the following inequality holds
$$\|u\|_{L^\infty (\R^N)} \le C(N, \beta, \lambda_{1,\beta}(A), |A|) \|u\|_{L^2(\R^N)}.$$
\end{lemma}
The proof of this lemma is given in \cite[Theorem 6.11]{BG15}.

\medskip

\noindent {\bf Analysis of the level sets.}
The proof of the Faber-Krahn inequality by Bossel and Daners is based on the following results involving the behavior of the level sets of eigenfunctions. Below, we recall from   \cite{BG10} those results in the context of nonsmooth sets.

For $t \in (0,||u||_{L^{\infty}}) $ we set:
\[
U_t =\{x \in A: \  u(x) >t\}.
\]
For a.e. $t>0$, the set $U_t$ has finite perimeter. For every nonnegative  Borel function and for the sets $U_t$ having finite perimeter we define:

\[
H_A(U_t, \varphi) =  \frac{1}{|U_t|} \left( \int_{\partial ^* U_t\setminus J_u} \varphi \,d \Hm - \int_{U_t} \varphi^2 \, dx\right) + \]
\[+  \frac{ \beta}{|U_t|} \left(\Hm(J_u  \cap \{ u^+ >t\}) +\Hm(J_u \cap \{ u^- >t\})\right)   .
\]

The following proposition hold true.
\begin{proposition}
\label{bfnt09}
Let $u$ be a minimizer of \eqref{weak}. Then
\[
\lambda_{1, \beta}(A)= H_A\left(U_t, \frac{|\nabla u|}{u}\right)
\]
for almost all $t \in (0,||u||_{L^\infty})$.
\end{proposition}
We refer to the orginal results of Bossel and Daners for this proposition in the context of smooth sets, to  Proposition 2.3 in \cite{BD} for the Lipschitz version. Proposition  \ref{bfnt09}  has been proved in the SBV context in  Lemma 6, Section 6 of \cite{BG10}.

\begin{proposition}
Let $\varphi$  be a nonnegative function such that $\varphi \in L^2(U_t)$ for every $t>0$ and let $u$ be a minimizer of \eqref{weak}. Set $w= \varphi - \displaystyle\frac{|\nabla u|}{u}$ and $F(t) = \displaystyle\int_{U_t} w \frac{|\nabla u|}{u} \, dx$, then $F$ is locally absolutely continuous and 
\begin{equation}
\label{ineq_eigen}
\lambda_{1, \beta}(A) \ge H_A\left(U_t, \varphi\right) + \frac{1}{t |U_t|} \frac{d}{dt}(t^2 F(t))
\end{equation}
for almost all $t \in (0,||u||_{L^\infty})$.
\end{proposition}
We refer to the orginal results of Bossel and Daners for this proposition in the context of smooth sets, to  Proposition 3.1 in \cite{BD} for the Lipschitz version and to  Lemma 7, Section 6 in \cite{BG10} for the general SBV-version. 
\medskip

\section{Proof of Theorem \ref{main}}
In this section we prove the main result of the paper, Theorem \ref{main}.
We follow the structure described in the introduction. 

\medskip
\noindent{\bf Step 1.} Let $\Omega $ be an open set of finite measure.  Then, for a fixed $k>0$, we consider the following auxiliary free discontinuity problem:
\begin{equation}
\label{auxiliary}
\min \{\lambda_{1, \beta}(A) + k |A|: A \subset \Omega, A  \mbox{ open}\}
\end{equation}
The purpose is to prove the existence of a solution to this problem and to prove that a nonnegative, $L^2$ normalized eigenfunction satisfies a uniform lower bound
$$\alpha ||u||_{L^2(\rn)} \le essinf_{x\in A} u(x)$$
with a constant $\alpha$ which is well controlled. 

In order to prove existence of a solution for  \eqref{auxiliary}, one has first to consider a relaxed problem in $SBV^\frac{1}{2}$ and, in a second step, to prove some regularity of the solution. 

We introduce the following relaxed problem
\begin{equation}
\label{bfnt04}
   \min \left \{ \begin{array}{ll}\displaystyle& \displaystyle \int_{\rn} |\nabla u|^2 \, dx + \beta  \displaystyle\int_{J_u}( |u^+|^2+|u^-|^2 )\, d\h +k |\{u>0\}|: u \in SBV^{  \frac 12} (\rn);  \\
\cr
& u=0 \> \hbox{a.e. in} \> \rn \setminus \Omega, \displaystyle  \int_{\rn} u^2 \, dx=1. \end{array} \right \}
\end{equation}
which, a priori, may lead to a lower value than the minimum in \eqref{auxiliary}. Indeed, a test function in \eqref{bfnt04} may not necessarily be a test function for the eigenvalue of some open set $A$ in  \eqref{auxiliary}. The solution of this problem is a function in $SBV^{  \frac 12} (\rn)$. In order to prove that it corresponds to an eigenfunction on an open set, the key points are to prove a non-degeneracy result and the (local) Ahlfors regularity of the jump set. 

We begin with the following.
\begin{proposition}
\label{bfnt10}
Problem \eqref{bfnt04} has a solution $u$. Moreover, there exists a constant $\alpha$  depending only on $N, \beta, \lambda_{1,\beta} (\Omega), |\Omega|,k$  such that
\begin{equation}
\label{uniformbd}
\alpha ||u||_{L^2(\rn)} \le essinf_{x\in \{u>0\}} u(x).
\end{equation}
\end{proposition}

 This result describes a sort of non-degeneracy of the solution $u$ at the free boundary. This nondegeneracy result  was first proved for  a generic minimizer of a free discontinuity problem of Robin type in \cite[Theorem 4.1]{BL14} and generalized in \cite[Theorem 6.13]{BG15} and \cite[Theorem 3.5]{BG15AIHP} in different directions.

A recent, simpler proof was given in 
 \cite[Theorem 3.2]{CK}, and has the advantage of providing an accurate estimate of the lower bound. Precisely, the following holds.
 \begin{lemma}
 \label{bfnt11}
 Let $u \in SBV^{  \frac 12} (\rn)$ be such that $|\{u>0\}|<+\infty$ and such that for almost every $t_0\ge t >0$
 $$F_{\beta,k}(u) \le F_{\beta,k}(u\cdot 1_{\{u >t\}}),$$
 where
 $$F_{\beta,k}(v)= \int_{\rn} |\nabla v|^2 dx+ \beta \int_{J_v}  |v^+|^2+|v^-|^2 \, d\h +k |\{v>0\}|.$$
 Then, there exists $\alpha >0$ depending only on $N,\beta,k,F(u), t_0$ such that
 $$\forall a.e. \, x \in \{u>0\}
 \; \;\;\; u(x) \ge \alpha.$$
 \end{lemma} 
 The proof of  Lemma \ref{bfnt11} follows step by step  \cite[Theorem 3.2]{CK}.
  
 \begin{proof}(of Proposition \ref{bfnt10})
 
 \noindent {\bf Existence of minimizers.} The existence of a minimizer of \eqref{bfnt04} comes by compactness arguments and the Ambrosio lower semicontinuity theorem. One follows Proposition 1, Section 3 from  \cite{BG10}. 
 
 \noindent {\bf Uniform non-degeneracy.} 
 Let $u$ be a minimizer of \eqref{bfnt04}.   We shall test its optimality with
 $$u_t=u\cdot 1_{\{u>t\}},$$
 rescaled such that its $L^2$-norm becomes equal to $1$. 
 
 Assume $|\Omega|=m$ and fix $\varepsilon _0>0$ be such that
 $$\frac{1}{1-\varepsilon _0^2 m}\le 1+2 \varepsilon _0^2 m,$$
 \begin{equation}
 \label{eps}
 \varepsilon _0^2 (\lambda_{1,\beta}(\Omega)+ k|\Omega|)<\frac 14. \end{equation}
 Then
 $$\frac {1}{\int_{\R^N} u_t^2 dx}\le 1+2 t^2 |\{0<u<t\}|,$$ and we can write from the optimality of $u$ in 
 \eqref{bfnt04}
 $$\int_{\rn} |\nabla u|^2 \, dx + \beta  \displaystyle\int_{J_u}( |u^+|^2+|u^-|^2 )\, d\h +k |\{u>0\}|\le$$
 $$\le  \Big (\int_{\rn} |\nabla u_t|^2 \, dx + \beta  \displaystyle\int_{J_{u_t}}( |u_t^+|^2+|u_t^-|^2 )\, d\h\Big )(1+2 t^2 |\{0<u<t\}|)  +k |\{u_t>0\}|\le$$
 $$\le \int_{\{u_t>0\}} |\nabla u_t|^2 \, dx + \beta  \displaystyle\int_{J_{u_t}}( |u_t^+|^2+|u_t^-|^2)\, d\h +\hskip 3cm $$
 $$\hskip 3cm + 4 t^2 |\{0<u<t\} | (\lambda_{1,\beta}(\Omega) +k|\Omega|)+k |\{u_t>0\}|.$$
 In the last inequality we assumed that 
\begin{equation}\label{bfnt12}
 \int_{\{u_t>0\}} |\nabla u_t|^2 \, dx + \beta  \displaystyle\int_{J_{u_t}}( |u_t^+|^2+|u_t^-|^2)\, d\h\le 2 (\lambda_{1,\beta}(\Omega) +k|\Omega|).\end{equation}
 
 For $0<t< \varepsilon_0\sqrt{\frac {k}{2}}= t_0(\lambda_{1,\beta}(\Omega),k,|\Omega|)$, using \eqref{eps},  \ we get
 \begin{equation}\label{bfnt13}
F_{\beta,{\frac{k}{2}}}(u)\le F_{\beta,{\frac{k}{2}}}(u\cdot 1_{\{u>t\}}).
\end{equation}
If \eqref{bfnt12} fails to be true, then \eqref{bfnt13} is trivially satisfied.

We conclude the proof on the basis of Lemma \ref{bfnt11}.
 \end{proof}

The nondegeneracy result proved above plays a crucial role in the existence result for the auxiliary free discontinuity problem \eqref{auxiliary}. The existence result has been proved in \cite{BG15}, in the case $\Omega=\R^N$. 
\begin{proposition}
\label{uniformth}
For every $k>0$ problem \eqref{auxiliary} admits a solution $A$. 
\end{proposition}
\begin{proof}

With respect to the case $\Omega=\R^N$ discussed in   \cite{BG15}, there is no any significant difference: 
\begin{itemize}
\item One takes a solution $u$ for \eqref{bfnt04}. 
\item The function $u$ belongs to $L^\infty$ and satisfies the lower estimate of Proposition \ref{bfnt10}.
\item Locally, in any ball contained in $\Omega$, the function $u$ is an almost quasi minimizer of the Mumford Shah functional (see \cite{BG15} and \cite{BL14}).  Consequently,  the jump set of the function $u$ is closed in $\Omega$. Identifying $A=\{u>0\}\setminus J_u$, we conclude the proof.
\item The set $A$ is connected, otherwise a connected component of the set $\{u>0\}\setminus J_u$ would give strictly lower energy.
\end{itemize} 
\end{proof}

\medskip

\noindent{\bf Step 2.}
Let $A$ be an open set of finite measure  and $B$ be a  ball having the  same volume as $A$. Assume its radius is $R$. We denote by $u_{B}$ a positive eigenfunction on $B$, which is radially symmetric. If $|x|=r$, we set
\[
\beta_r=  \varphi_B(x) = \frac{|\nabla u_B|}{u_B}(x)
\]
is well defined and strictly increasing for $r \in (0,R)$ with $\beta_0=0$ and  $\beta_R=\beta$ (see Proposition 4. 2 in \cite{BD}).
Given $t \in (0,||u||_{L^\infty})$ we define $r(t)$ to be the radius of the ball with the same center as $B$ and  the same volume as $U_t$.  Then we define

\begin{equation}
\label{varfi}
\varphi(x) = \beta_{r(t)}
\end{equation}
if $ x \in \Omega$ and $u(x)=t$.

\begin{lemma}
\label{step1}
Let $A$ be an open, connected set of finite measure and  $u$ a nonnegative, $L^2$-normalized  eigenfunction corresponding to $\lambda_{1,\beta}(A)$. Then
\begin{equation}
\label{eigenineq}
\lambda_{1, \beta} (A)-\lambda_{1, \beta} (B)\ge   \frac \beta2 (ess \inf_{x\in A} u(x))^2\big  ({\mathcal H}^{N-1}(\partial^* A)-{\mathcal H}^{N-1}(\partial B)\big ),
\end{equation}
\end{lemma}
As a consequence of the quantitative isoperimetric inequality of \cite{FMP08}, the following   corollary holds. Let us denote  $u_{min} = ess \inf_{x\in A} u(x)$.
\begin{corollary}
Under the hypotheses of Lemma \ref{step1} we have
\begin{equation}
\label{eigenineq.1}
\left[  \lambda_{1,\beta}(A) - \lambda_{1,\beta}(B) \right] \ge  C(N)|A|^\frac{N-1}{N} {u_{min}^2} \beta  {\mathcal A (A)^2},
\end{equation}
for some constant $C(N)$ depending only on the dimension of the space.

\end{corollary}

\begin{proof} (of Lemma \ref{step1}.) Before beginning the proof, let us distinguish between the cases $ess \inf_{x\in A} u(x)=0$ and $ess \inf_{x\in A} u(x)>0$. If $ess \inf_{x\in A} u(x)=0$, then inequality \eqref{eigenineq} is true, as a consequence of the convention $0\cdot \infty =0$. If $ess \inf_{x\in A} u(x)>0$ we claim that $A$ is a set with finite perimeter, i.e. ${\mathcal H}^{N-1}(\partial^* A)<+\infty$. Indeed, since $ess \inf_{x\in A} u(x)>0$, the function  $1_A$ is an element of $SBV^\frac 12(\rn)$ and in particular that $1_A$ is an element of $SBV(\R^N)$. 

From Proposition 2.2 we have, for almost all $t \in (0,||u||_{L^\infty})$, 
\begin{equation}
\label{eigenball}
\lambda_{1,\beta}(B) = H_B(B_{r(t)}, \varphi_B)= \frac{1}{|B_{r(t)}|}\left ( \int_{\partial B_{r(t)}}  \frac{|\nabla u_B|}{u_B}\, d \Hm - \int_{B_{r(t)}} \frac{|\nabla u_B|^2}{u_B^2} \, dx  \right),
\end{equation}

 Proposition 2.3 with $\varphi$ defined in \eqref{varfi} and equality \eqref{eigenball}, ensures that for almost all $t \in (0,||u||_{L^\infty})$
\begin{equation}
\label{ineq0}
\begin{array}{l}
 \lambda_{1,\beta} (\Omega) - \lambda_{1,\beta}(B) \ge   \displaystyle \frac{1}{t |U_t|} \frac{d}{dt}(t^2 F(t)) +  \displaystyle   \frac{1}{|U_t|}  \left( \displaystyle\int_{\partial ^* U_t\setminus J_u} \varphi \,d \Hm - \displaystyle \int_{U_t} \varphi^2 \, dx\right)  +   \\ \hfill
+ \displaystyle \frac{1}{|U_t|} \left( \beta  \big (\Hm(J_u \cap \{ u^+ >t\}) + \Hm(J_u \cap \{ u^- >t\}) \big  ) \right)  \\ \hfill
 -  \displaystyle\frac{1}{|B_{r(t)}|}\left ( \displaystyle\int_{\partial B_{r(t)}}  \displaystyle\frac{|\nabla u_B|}{u_B}\, d \Hm - \displaystyle\int_{B_{r(t)}} \displaystyle\frac{|\nabla u_B|^2}{u_B^2} \, dx  \right).
\end{array}
\end{equation} 

Now recalling that  for  $t \in (0,||u||_{L^\infty})$ 
\begin{itemize}
\item $\varphi(x) = \beta_{r(t)}= \displaystyle\frac{|\nabla u_B|}{u_B}(r(t))$ on $ \partial ^* U_t\setminus \partial ^* A$, 
\smallskip
\item $ \displaystyle\frac{|\nabla u_B|}{u_B}(r(t))=\beta_{r(t)} <\beta$, 
\smallskip
\item $|U_t|=|B_{r(t)}|$,
\smallskip
\item $\displaystyle\int_{U_t} \varphi^2 \, dx= \displaystyle\int_{B_{r(t)}} \displaystyle\frac{|\nabla u_B|^2}{u_B^2} \, dx$ 
\end{itemize}

using \eqref{ineq0} we get for almost every $ t \in (0,||u||_{L^{\infty}})$
\begin{equation}
\label{firstineq}
t|U_t| \left[  \lambda_{1,\beta} (A) - \lambda_{1,\beta}(B) \right] \ge t \frac{|\nabla u_B|}{u_B}(r(t)) \left[ \h(\partial ^*U_t) -\h (\partial B_{r(t)})\right] + 
\displaystyle\frac{d}{dt}(t^2 F(t)).
\end{equation}

Integrating the above inequality and using the fact that the function $t^2F(t)$ vanishes at the ends, we get
$$
\int_0^{||u||_{L^{\infty}}}t|U_t| \left[  \lambda_{1,\beta}(A) - \lambda_{1,\beta}(B) \right]\, dt \ge\hskip 3cm $$
$$\hskip 3cm \ge \int_0^{||u||_{L^{\infty}}} t \frac{|\nabla u_B|}{u_B}(r(t)) \left[ \h(\partial ^*U_t) -\h(\partial B_{r(t)})\right] \,dt.
$$

\end{proof}
Following  the sharp quantitative   isoperimetric inequality proved in \cite{FMP08}
$$
\frac{||u||_{L^2}^2}{2} \left[  \lambda_{1,\beta}(A) - \lambda_{1,\beta}(B) \right] \ge \frac{u_{min}^2}{2} \beta \left[ \Hm(\partial ^*A) - \Hm(\partial B)\right] \ge $$
$$ \hskip 5cm  \ge C(N)|A|^\frac{N-1}{N} \frac{u_{min}^2}{2} \beta  {\mathcal A (A)^2}.
$$

\noindent {\bf Step 3.}
Let us consider $\Omega \subset \R^N$, an open set of finite measure. In order to avoid confusion, below denote and $B^{|\Omega|}$ a ball of the same measure as $\Omega$. We assume that $\lambda_{1,\beta}(\Omega) \le 2\lambda_{1,\beta}(B^{|\Omega|})$, otherwise the quantitative inequality holds with a suitable constant.

In view Theorem 5.1,  Lemma 5.2 in \cite{BG15} and Lemma \ref{convexity}, there exists a constant $k=k(|\Omega|,N, \beta)$ such that
\[
\lambda_{1,\beta}(A) + k |A| \ge \lambda_{1,\beta}(B^{|\Omega|}) + k |B^{|\Omega|}| 
\] 
for all $A\subset \R^N$ open set of finite measure.

For the set $\Omega$, 
we denote by $\Ao$   the optimal set solving \eqref{auxiliary} with the the constant $k$ defined above, and by $u_{\Ao}$ a nonnegative eigenfunction associated to $\lambda_{1,\beta}(A_{opt})$ normalized in such a way that $||u_{\Ao}||_{L^2}=1$. We denote  $B^{\Ao}$ a ball of the same volume as $\Ao$ achieving the Fraenkel asymmetry .

By Faber-Krahn inequality \eqref{FK}, the choice of $k$ and Corollary \ref{uniformconvexity}, we have that there exists a positive constant $C(N, |\Omega|)$ such that

\begin{equation}
\label{convex}
\begin{array}{ll}
\lambda_{1,\beta}(\Ao) + k|\Ao| &\ge \lambda_{1,\beta}(B^{\Ao}) + k|B^{\Ao}|   \\
\cr 
& \ge \lambda_{1,\beta}(B^{|\Omega|}) + k|B^{|\Omega|}|  
 +C(N,|\Omega|)(|B^{|\Omega|}|-|B^{\Ao}|)^2.
\end{array}
\end{equation}
Moreover by  Lemma \ref{step1}, Proposition \ref{bfnt10}  and the minimality of $\Ao$ we have that there exists a positive constant $C(N, \beta,|\Omega|)$ such that
\begin{equation}
\label{minimality}
\begin{array}{ll}
\lambda_{1,\beta}(\Omega) +k |\Omega| &\ge\lambda_{1,\beta}(\Ao) + k|\Ao|   \ge \\
\cr
&\ge \lambda_{1,\beta}(B^{|\Ao|}) + k|B^{\Ao}| +C(N, \beta) \displaystyle \frac{|\Ao \Delta B^{\Ao}|^2}{|\Ao|^{1+\frac 1N}} \\
&\ge \lambda_{1,\beta}(B^{|\Ao|}) + k|B^{\Ao}| +C(N, \beta, |\Omega| ) \displaystyle |\Ao \Delta B^{\Ao}|^2
\end{array}
\end{equation}
and then joining \eqref{convex} and \eqref{minimality}
\begin{equation}
\label{final}
\begin{array}{ll}
\lambda_{1,\beta}(\Omega) &\ge  \lambda_{1,\beta}(B^{|\Omega|})+ C(N, \beta, |\Omega|)(|B^{|\Omega|} |-|B^{\Ao}| +|\Ao \Delta B^{\Ao}|)^2 \\ 
\cr
&\ge 
 \lambda_{1,\beta}(B^{|\Omega|})+ C(N, \beta, |\Omega|) \mathcal{A}^2(\Omega)
\end{array}
\end{equation}
where in the last inequality we used the very definition of Fraenkel asymmetry \eqref{fraenkel}, the optimality of $B_{\Ao}$ and the fact that $\Ao\subset \Omega$.

The constants  $C(N, \beta, |\Omega|)$ above do not depend explicitly on $k$ and $\lambda_{1,\beta}(\Omega)$ since $k$ depends only on $N, \beta$ and $|\Omega|$ and we have chosen to work with sets $\Omega$ such that $\lambda_{1,\beta}(\Omega) \le 2\lambda_{1,\beta}(B^{|\Omega|})$, so that the upper bound of $\lambda_{1,\beta}(\Omega)$ depends only on $N, \beta$ and $|\Omega|$.

\section{Further remarks}
We have a quite explicit dependence of the constant $c$ in \eqref{bfnt01}. A natural question is whether we can recover the inequality of \cite{BDe} when $\beta$ goes to $+\infty$. In fact, the important term to understand is $\alpha^2 \beta$, where $
\alpha$ is the constant given in Proposition \ref{bfnt10}.  Since in our use of this term, $k$ and $\lambda_{1,\beta}(\Omega)$ are controlled by $N, |\Omega|$ and $\beta$, at fixed measure,  the constant $\alpha$ depends fundamentally only on $\beta$ and $N$. A simple analysis shows that $\alpha^2 \beta$ vanishes as $\beta \rightarrow +\infty$.
Thus the answer is negative and can be explained by the method we use to prove Theorem \ref{main}, which is essentially different from the one in \cite{BDe}.  In our case, the effort is concentrated to compare the difference of the eigenvalues with the difference of the perimeters and to relay on the result of \cite{FMP08}, not to select sets which are close in $C^1$ distance to the ball. Our strategy, if used for Dirichlet boundary conditions, would merely lead to a non sharp power of the Fraenkel asymmetry. 

The power $2$ on the Fraenkel asymmetry in \eqref{bfnt01} is sharp. This value coincides with the powers appearing in the quantitative isoperimetric inequality, and the quantitative Faber-Krahn inequality for the first Dirichlet eigenvalue. In order to observe the sharpness, one can use the following perturbation of the ball in ellipsoids: let
$$E_\varepsilon =\Big \{(x_1,\dots, x_N) \in \R^N: \frac{x_1^2}{(1+\varepsilon)^2}+ \frac{x_1^2}{(1-\varepsilon)^2}+x_3^2+\dots x_N^2=1 \Big\}.$$
We denote
$$\tilde E_\varepsilon = (1-\varepsilon^2)^{-\frac 1N} E_\varepsilon,$$
the ellipsoid with the same volume as $B_1$.

We observe that
$${\mathcal A} (\tilde E_\varepsilon) \sim  c_1 \varepsilon^2,$$
$$\lambda_{1, \beta} (\tilde E_\varepsilon)-\lambda_{1, \beta} (B_1)\le c_2 \varepsilon^2.$$

An interesting question is whether one could avoid technicalities of SBV spaces in order to prove the quantitative inequality of Theorem \ref{main} provided we restrict to the class of Lipschitz sets. At this moment, it seems quite difficult to answer positively. To be precise, Steps 2 and 3 can be easily carried in the class of Lipscthiz sets. The difficulty comes from  the selection procedure of the set $A$ in Step 1. At this moment the authors do not see how to deal with Step 1 in a direct manner, remaining in the class of Lipschitz sets. 

Another question concerns the pertinence of the definition of the fundamental Robin eigenvalue in a nonsmooth open set with finite measure. We followed in this paper the framework developed in \cite{BG10,BG15}. The eigenvalue defined in this way  emerges as a natural relaxation of the eigenvalue and, in particular, coincides with the classical one in the case of Lipschitz sets and Lipschitz sets having $(N-1)$ dimensional Lipschitz cracks.  

An alternative way to define the eigenvalue in a nonsmooth, bounded, open set is to use the framework introduced by Daners in \cite{da00}. Given $A\subset \R^N$, 
the eigenvalue is defined to be
\begin{equation}
\label{bfnt15}
\overline \lambda_{1,\beta}  (A)=\inf \Big \{ \frac{\displaystyle \int_A |\nabla u|^2dx +\beta \int_{\partial A} u^2 d\Hm}{\displaystyle \int_A u^2 dx} : u \in H^1(A)\cap C(\overline A) \cap C^\infty (\Omega) \Big \}.
\end{equation}
One has in general that for a given open set of finite measure, the eigenvalue defined in Definition \ref{bfnt07} is not larger than the one defined by \eqref{bfnt15}   (see the remark at the end of Section 3 in \cite{BG10})
$$ \lambda_{1,\beta}  (A)\le \overline \lambda_{1,\beta}  (A).$$
In particular, this definition of the eigenvalue does not  take into account that an eigenfunction may have two possibly different traces on a crack. Although this eigenvalue does not capture some fine geometric properties, it has the advantage to be  the lowest eigenvalue corresponding to a suitably defined Laplace operator, which inherits good functional analytic properties. This is not the case of the fundamental eigenvalue defined in \eqref{weak}. 

As a conclusion, both the Faber-Krahn and and the quantitative Faber-Krahn inequalities hold as well for arbitrary open sets, if  one uses definition \eqref{bfnt15}.

\end{document}